\documentclass[oneside,11pt,a4paper,reqno,ps]{article}
\usepackage[a4paper,top=5cm,bottom=5cm,left=3cm,right=3cm,
bindingoffset=5mm]{geometry}
\usepackage[latin1]{inputenc}
\usepackage[english]{babel}
\usepackage[T1]{fontenc}
\usepackage{syntonly}
\usepackage{graphicx}
\usepackage{amssymb}	
\usepackage{amsthm}
\usepackage{microtype}
\usepackage{amsmath}
\usepackage{braket}
\usepackage{eurosym}
\usepackage{amssymb}
\usepackage{amstext}
\usepackage{color}
\usepackage{amsmath,amsthm,amssymb,amscd,epsfig,esint, mathrsfs,graphics,color}
 \usepackage{wrapfig}
\usepackage{verbatim}
\newcommand{\numberset}{\mathbb}
\newcommand{\R}{\numberset{R}}
\newcommand{\N}{\numberset{N}}

\theoremstyle{plain}
\newtheorem{thm}{Theorem}[section]

\newtheorem{lemma}[thm]{Lemma}

\theoremstyle{definition}

\def\Xint#1{\mathchoice 
	{\XXint\displaystyle\textstyle{#1}}%
	{\XXint\textstyle\scriptstyle{#1}}%
	{\XXint\scriptstyle\scriptscriptstyle{#1}}%
	{\XXint\scriptscriptstyle\scriptscriptstyle{#1}}%
	\!\int} 
\def\XXint#1#2#3{{\setbox0=\hbox{$#1{#2#3}{\int}$} 
		\vcenter{\hbox{$#2#3$}}\kern-.5\wd0}} 
\def\Mint{\Xint -}

\def\R{\mathbb{R}}
\numberwithin{equation}{section} \makeatletter

\renewcommand{\p@enumi}{\thesection.}
\title{\textbf{Regularity for minimizers of a class of non-autonomous functionals with sub-quadratic growth}}
\author{Andrea Gentile }

\begin{document}
	
	\maketitle
	
\begin{abstract}
	\noindent We consider functionals of
	the form
	$$
	\mathcal{F}( v, \Omega )= \int_{\Omega} \! f(x,Dv(x)) \, dx,
	$$
	with convex integrand  with respect
	to the gradient variable, assuming that the function that measures
	the oscillation of the integrand with respect to the $x$ variable
	belongs to a suitable Sobolev space $W^{1,q}$. \\
	We prove a result of higer differentiability for the minimizers.
	We also infer a result of Lipschitz regularity of minimizers if $q>n$, and a result of higher integrability for the gradient if $q=n$.
	\noindent The novelty here is that we deal with integrands satisfying subquadratic growth conditions with respect to gradient variable.
\end{abstract}


\bigskip

\bigskip

\section{Introduction}
In this paper, we consider integral functionals of the type

\begin{equation}\label{functional}
\mathcal{F}(v, \Omega)=\int_{\Omega}f\left(x, Dv(x)\right)dx,
\end{equation}

where $\Omega\subset\R^n$ is a bounded open set,  $f:\Omega\times\R^{N\times n}\to\R$ is a Carath\'{e}odory map, such that $\xi\mapsto f(x, \xi)$ is of class $C^2(\R^{N\times n})$ for a.e. $x\in\Omega$, and for an exponent $p\in(1, 2)$ and constants  $\ell_1, \ell_2>0, L_1, L_2\ge0,$ and a parameter $\mu\ge0$ the following conditions are satisfied:

\begin{equation}\label{f1}
\ell_1(\mu^2+|\xi|^2)^\frac{p}{2}\le f(x, \xi)\le \ell_2(\mu^2+|\xi|^2)^\frac{p}{2},
\end{equation}

\begin{equation}\label{SecondDerivativeGrowth}
L_1\left(\mu^2+\left|\xi\right|^2\right)^\frac{p-2}{2}\left|\eta\right|^2\le\left<D_{\xi\xi}f(x, \xi)\eta, \eta\right>\le L_2\left(\mu^2+\left|\xi\right|^2\right)^\frac{p-2}{2}\left|\eta\right|^2.
\end{equation}

for almost every $x$ in $\Omega$, and for all $\xi, \eta$ in $\R^{N\times n}$.
For what concerns the dependence of the energy density on the $x$-variable, we shall assume that the function $D_\xi f(x, \xi)$ is weakly differentiable with respect to $x$ and that $D_x(D_\xi f) \in L^q(\Omega\times\R^{N\times n})$, for some $q\ge n$.\\
This is equivalent to assume that there exists a nonnegative  function $g\in L^q_\text{loc}(\Omega)$ such that

\begin{equation}\label{f4}
\left|D_x\left(D_\xi f(x, \xi)\right)\right|\le g(x)\left(\mu^2+|\xi|^2\right)^\frac{p-1}{2}
\end{equation} 

for  all $\xi\in\R^{N\times n}$ and for almost every $x\in \Omega$.\\
In order to avoid the irregularity phenomena that are peculiar of the vectorial minimizers (see \cite{DeGiorgi}, \cite{Sverak-Yan}), we shall assume that 

\begin{equation}\label{radial}
f(x, \xi)=k(x, \left|\xi\right|)
\end{equation}

with

\begin{equation}\label{xidependence}
k(x, \cdot)\in C^2(\R) \text{ if }\mu>0\qquad\text{or}\qquad k(x, \cdot)\in C^2(\R\setminus\set{0})\text{ if }\mu=0,
\end{equation}

for almost every $x \in \Omega$.\\
\noindent The regularity properties of minimizers of such integral
functionals   have been widely
investigated in case the energy density $f(x,\xi)$ is continuous as a function of the $x$-variable, both in the superquadratic and in the subquadratic growth case. Actually, the partial continuity of the vectorial
minimizers can be obtained with a quantitative modulus of continuity
that depends on the modulus of continuity of the coefficients (see
for example \cite{AF, fh, gm} and the monographs \cite{19,23}
for a more exhaustive treatment). For regularity results under general growth conditions, that of course include the superquadratic and  the subquadratic ones, we refer to \cite{diestrver09,diestrver11, FuscoSbordone, LeonettiMascoloSiepe}.

Recently, there has been an increasing interest in the study of the
regularity when the oscillation of $f(x,\xi)$ with respect to the $x$-variable is controlled through a coefficient that belongs to a suitable Sobolev class of integer or fractional order and the assumptions \eqref{f1}--\eqref{f4} are satisfied with an exponent $p\ge 2$.\\

Actually, it has been shown that the weak
differentiability of the partial map $x\mapsto f(x,\xi)$ transfers to the
gradient of the minimizers of the functional \eqref{functional} (see
\cite{5,EleMarMas, EleMarMas2, GP, KristensenMingione,32}) as well as to the gradient  of the
solutions of non linear elliptic systems (see \cite{2,6,CGP,CuMR,Gio,KuusiMingione,33}) and of non linear systems with degenerate ellipticity in case $p\ge2$. (see \cite{Gio}).\\
It is worth mentioning that the continuity of the coefficients is not sufficient to establish the higher differentiability of integer order of the minimizers.

As far as we know, no regularity results are available for vectorial minimizers nor to establish their Lipschitz continuity under the so-called subquadratic growth conditions, i.e. when the  assumptions \eqref{f1}--\eqref{f4} hold true for an exponent $1<p\le 2$ in case of Sobolev coefficients.\\
The aim of this paper is to prove that, assuming $g\in L^q_{\mathrm{loc}}(\Omega)$, with $q\ge n$, any local minimizer $u\in W^{1, p}_{\mathrm{loc}}(\Omega)$ of the functional \eqref{functional} is higher differentiable, that is $u\in W^{2, p}_{\mathrm{loc}}(\Omega)$. Moreover, if $q>n$, we establish the Lipschitz continuity of the local minimizers, and for $q=n$ we prove that the gradient of $u$ is in $L^r_{\mathrm{loc}}(\Omega)$ for any $r\in(1, \infty)$. We will use the following auxiliary function that, as a function of the gradient of a local minimizer of the functional \eqref{functional}, will be the main object of our results.

\begin{equation}\label{Hdef}
H(\xi)=\left(\mu^2+\left|\xi\right|^2\right)^\frac{1}{2}, \qquad\forall\xi\in\R^{N\times n}.
\end{equation}

More precisely, our main results are the following.

\begin{thm}\label{qgrn}
	Let $u\in W^{1,p}_{\mathrm{loc}}(\Omega)$ be a local minimizer of the functional \eqref{functional}, under the assumptions \eqref{f1}--\eqref{xidependence}.\\
	If $q>n$, then $u\in W^{2,p}_{\mathrm{loc}}(\Omega)$ and $H(Du)\in L^{\infty}_{\mathrm{loc}}(\Omega)$.\\
	Moreover, there exist two constants $c_1, c_2>0$, depending on $n, N, p, q, L_1, L_2, \Arrowvert g\Arrowvert_{L^q(B_R)}$, such that the following estimates hold:
	
	\begin{equation}\label{Linftyclaim}
	\Arrowvert H\left(Du\right)\Arrowvert_{L^\infty\left(B_{\frac{R}{2}}\right)}\le c_1 \Arrowvert H\left(Du\right)\Arrowvert_{L^{p}\left(B_{R}\right)},
	\end{equation}
	
	and
	
	\begin{equation}\label{LpD2uestimateclaimqgrnthm}
	\int_{B_{\frac{R}{2}}}\left|D^2u(x)\right|^pdx\le c_2\cdot\left(\int_{B_{R}}H^p\left(Du(x)\right)dx\right),
	\end{equation}
	
	for every ball $B_R$ such that $B_R\Subset\Omega$.
\end{thm}

In the critical case $q=n$, we have the following. 

\begin{thm}\label{qeqn}
	Let $u\in W^{1,p}_{\mathrm{loc}}(\Omega)$ be a local minimizer of the functional \eqref{functional}, under the assumptions \eqref{f1}--\eqref{xidependence}.\\
	If $q=n$, then, for any $1<r<\infty$, $H(Du)\in L^{r}_{\mathrm{loc}}(\Omega)$, and there is a constant $c_1=c_1(n, N, p, r, L_1, L_2, \Arrowvert g\Arrowvert_{L^n(B_R)})\ge0$, such that, for every $R>0$ such that $B_R\Subset\Omega$, the following estimate holds
	
	\begin{equation}\label{Lrestimate-claimthm}
	\left(\int_{B_{\frac{R}{2}}}H^{r}\left(Du(x)\right)dx\right)^{\frac{1}{r}}\le c_1\cdot\left(\int_{B_{R}}H^{p}\left(Du(x)\right)dx\right)^{\frac{1}{p}}.
	\end{equation}
	
	Moreover, $u\in W^{2,p}_{\mathrm{loc}}(\Omega)$, and there exists a constant $c_2=c_2(n, N, p, L_1, L_2, \Arrowvert g\Arrowvert_{L^n(B_R)})\ge0$ such that
	
	\begin{equation}\label{LpD2uestimateclaimqeqnthm}
	\int_{B_{\frac{R}{2}}}\left|D^2u(x)\right|^pdx\le c_2\cdot\left(\int_{B_{R}}H^p\left(Du(x)\right)dx\right).
	\end{equation}
\end{thm}

It is worth mentioning that, in this case, the partial map $x\mapsto D_{\xi}f(x, \xi)$ needs not to be continuous. Actually, by the Sobolev embedding theorem we have that it belongs to the space $VMO$ of function with vanishing mean oscillation (see \cite{23} for the precise definition). The regularity of solutions to PDEs with $VMO$ coefficients goes back to \cite{Iwaniec-Sbordone} and \cite{Kinnunen-Zhou}.\\
Estimate \eqref{Lrestimate-claimthm} can be interpreted as an extension of the result in \cite{Kinnunen-Zhou} that concerns the $p-$Laplace operator to more general operator with sub-quadratic growth.\\
The proofs of our results are achieved combining suitable a priori estimates with an approximation argument. First of all, making the a priori assumption that $u\in W^{2,2}_{\mathrm{loc}}\left(\Omega\right)\cap W^{1,\infty}_{\mathrm{loc}}\left(\Omega\right)$, we will use Moser's iterative technique (see \cite{CGGP2}) to find an a priori estimates for the $L^\infty-$norm of $H\left(Du\right)$ in case $q>n$, and an a priori estimate for the $L^r-$norm of $H(Du)$ for any $1<r<\infty$ if $q=n$. We will also find a new a priori estimate for the $L^p-$norm of the second derivatives of $u$ that impreves that established in \cite{Gentile}.\\
After that, by approximation, we will use these a priori estimates to prove that a minimizer $u\in W^{1,p}_{\mathrm{loc}}(\Omega)$ is actually in $W^{2,p}_{\mathrm{loc}}(\Omega)$ and, if $q>n$, then $H(Du)\in L^{\infty}_{\mathrm{loc}}(\Omega)$, while, if $q=n$,
$H(Du)\in L^r_{\mathrm{loc}}(\Omega)$ for all $1<r<\infty.$

In \cite{Gentile}, making some weaker assumptions about the dependence of $f$ on the $\xi-$variable, more precisely, $\xi\mapsto f(x, \xi)$ is of class $C^1(\R^{N\times n})$, and instead of \eqref{SecondDerivativeGrowth}, for some $\alpha>0,$

\begin{equation}\label{f2}
\left<D_\xi f(x, \xi)-D_\xi f(x, \eta), \xi-\eta\right>\ge\alpha\left(\mu^2+|\xi|^2+|\eta|^2\right)^\frac{p-2}{2}|\xi-\eta|^2,
\end{equation}

for every $\xi, \eta\in \R^{N\times n}$ and for almost every $x\in\Omega$, and assuming that, instead of \eqref{f4}, the following condition

\begin{equation}\label{f4'}
|D_\xi f(x, \xi)-D_\xi f(y, \xi) |\le \left(g(x)+g(y)\right)\left|x-y\right|\left(\mu^2+|\xi|^2\right)^\frac{p-1}{2},
\end{equation} 

holds for a function $g\in L^q_{\mathrm{loc}}(\Omega)$ with $q\ge\frac{2n}{p}$, an a priori estimate for the $W^{2, p}-$norm of the local minimizers of the functional \eqref{functional} has been proved.

\section{Notations and preliminaries}

In this section we list the notations that we use in this paper and recall some tools that will be useful to prove our results.\\
We shall follow the usual convention and denote by $C$ or $c$ a general constant that may vary on different occasions, even within the same
line of estimates. Relevant dependencies on parameters and special
constants will be suitably emphasized using parentheses or
subscripts. All the norms we use on $\R^n$, $\R^N$ and $\R^{N\times n}$ will be
the standard Euclidean ones and denoted by $| \cdot |$ in all cases.
In particular, for matrices $\xi$, $\eta \in \R^{N\times n}$ we write $\langle
\xi, \eta \rangle : = \text{trace} (\xi^T \eta)$ for the usual inner
product of $\xi$ and $\eta$, and $| \xi | : = \langle \xi, \xi
\rangle^{\frac{1}{2}}$ for the corresponding Euclidean norm. When $a
\in \R^N$ and $b \in \R^n$ we write $a \otimes b \in \R^{N\times n}$ for the
tensor product defined
as the matrix that has the element $a_{r}b_{s}$ in its $r$-th row and $s$-th column.\\
For  a  $C^2$ function $f \colon \Omega\times\R^{N\times n} \to \R$, we write
$$
D_\xi f(x,\xi )[\eta ] := \frac{\rm d}{{\rm d}t}\Big|_{t=0} f(x,\xi
+t\eta )\quad \mbox{ and } \quad D_{\xi\xi}f(x,\xi )[\eta ,\eta ] :=
\frac{\rm d^2}{{\rm d}t^{2}}\Big|_{t=0} f(x,\xi +t\eta )
$$
for $\xi$, $\eta \in \R^{N\times n}$ and for almost every $x\in \Omega$.\\
With the symbol $B(x,r)=B_r(x)=\{y\in
\R^n:\,\, |y-x|<r\}$, we will denote the ball centered at $x$ of
radius $r$ and
$$(u)_{x_0,r}= \Mint_{B_r(x_0)}u(x)\,dx,$$
stands for the integral mean of $u$ over the ball $B_r(x_0)$. We
shall omit the dependence on the center  when it is clear from the context.

\section{A priori estimates}\label{A priori section}

Our first step is to prove some a priori estimates. More precisely, making a distinction between the cases $q>n$ and $q=n$ in the assumption \eqref{f4}, and being $H$ the function defined by \eqref{Hdef}, we want to prove the following claims.

\begin{lemma}\label{aprioriqgrn}
	Let $u\in W^{2,2}_{\mathrm{loc}}(\Omega)\cap W^{1,\infty}_{\mathrm{loc}}(\Omega)$ be a local minimizer of the functional \eqref{functional}, under the assumptions \eqref{f1}--\eqref{xidependence}. If $q>n$, then there exist two constants $c_1, c_2\ge0$, depending on $n, N, p, q, L_1, L_2, \Arrowvert g\Arrowvert_{L^q(B_R)}$, such that the following estimates hold:

\begin{equation}\label{aprioriLinftyclaim}
\Arrowvert H\left(Du\right)\Arrowvert_{L^\infty\left(B_{\frac{R}{2}}\right)}\le c_1 \Arrowvert H\left(Du\right)\Arrowvert_{L^{p}\left(B_{R}\right)},
\end{equation}

and

\begin{equation}\label{LpD2uestimateclaimqgrn}
\int_{B_{\frac{R}{2}}}\left|D^2u(x)\right|^pdx\le c_2\int_{B_{R}}H^p\left(Du(x)\right)dx,
\end{equation}

for every ball $B_R$ such that $B_R\Subset\Omega$.
\end{lemma}

\begin{lemma}\label{aprioriqeqn}
	Let $u\in W^{2,2}_{\mathrm{loc}}(\Omega)\cap W^{1,\infty}_{\mathrm{loc}}(\Omega)$ be a local minimizer of the functional \eqref{functional}, under the assumptions \eqref{f1}--\eqref{xidependence}. If $q=n$, then, for any $1<r<\infty$ there is a constant $c_1\ge0$, depending on $n, N, p, r, L_1, L_2, \Arrowvert g\Arrowvert_{L^n(B_R)}$, such that, for every $R>0$ such that $B_R\Subset\Omega$, the following estimate holds
	
\begin{equation}\label{Lrestimate-claim}
\left(\int_{B_{\frac{R}{2}}}H^{r}\left(Du(x)\right)dx\right)^{\frac{1}{r}}\le c_1\left(\int_{B_{R}}H^{p}\left(Du(x)\right)dx\right)^{\frac{1}{p}}.
\end{equation}

Moreover, there exists a constant $c_2=c_2(n, N, p, L_1, L_2, \Arrowvert g\Arrowvert_{L^n(B_R)})\ge0$ such that
	
\begin{equation}\label{LpD2uestimateclaimqeqn}
	\int_{B_{\frac{R}{2}}}\left|D^2u(x)\right|^pdx\le c_2\int_{B_{R}}H^p\left(Du(x)\right)dx.
\end{equation}
\end{lemma}
\subsection{The case $q>n$: proof of Lemma \ref{aprioriqgrn}}

\begin{proof}[Proof of Lemma \ref{aprioriqgrn}]
Our starting point is the the Second Variation of the functional $\mathcal{F}$.
Let us consider a test function $\varphi=D\psi$, with $\psi\in C^\infty_0(\Omega)$, and put $\varphi$ in the Euler-Lagrange equation of $\mathcal{F}$, so we have

\begin{equation}
\int_{\Omega}\left<D_\xi f(x, Du(x)), D^2\psi(x)\right>dx=0,
\end{equation}

and an integration by parts yields

\begin{equation}\label{SecondVariation1}
\int_{\Omega}\left<D_x\left(D_\xi f(x, Du(x))\right), D\psi(x)\right>=0,
\end{equation}

that is

\begin{equation}\label{SecondVariation2}
\int_{\Omega}\left<D_{x\xi}f(x, Du(x))+D_{\xi\xi}f(x, Du(x))D^2u(x), D\psi(x)\right>=0.
\end{equation}

Now, for a point $x_0\in\Omega$, we set $B_R=B_R(x_0)$, where $0<\rho<R<d(\partial\Omega, x_0)$, and we choose a cut-off function $\eta\in C^{\infty}_0(B_R)$ such as $0\le\eta\le1$, $\eta\equiv1$ on $B_\rho$, and $\left|D\eta\right|\le\frac{c}{R-\rho}$ for a constant $c>0$. The a priori assumption $u\in W^{1, \infty}_{\mathrm{loc}}(\Omega)\cap W^{2,2}_{\mathrm{loc}}(\Omega)$ allows as to consider, for $\gamma\ge0$, the test function $\psi=\eta^2\left(\mu^2+\left|Du\right|^2\right)^\frac{\gamma}{2} Du$ in the equation \eqref{SecondVariation2}. Computing the derivatives of $\psi$, we get

\begin{align*}
D\psi&=2\eta\left(\mu^2+\left|Du\right|^2\right)^\frac{\gamma}{2}D\eta\otimes Du+\frac{\gamma}{2}\eta^2\left(\mu^2+\left|Du\right|^2\right)^\frac{\gamma-2}{2}D\left(\left|Du\right|^2\right)\otimes Du \\&+\eta^2\left(\mu^2+\left|Du\right|^2\right)^\frac{\gamma}{2}D^2u,
\end{align*}

and the equation \eqref{SecondVariation2} becomes

\begin{align}\label{SecondVariation3}
0&=2\int_{B_R}\left<D_{x\xi} f\left(x, Du(x)\right), \eta(x)\left(\mu^2+\left|Du(x)\right|^2\right)^\frac{\gamma}{2} D\eta(x)\otimes Du(x)\right>dx\notag\\
&+\frac{\gamma}{2}\int_{B_R}\left<D_{x\xi}f\left(x, Du(x)\right), \eta^2(x)\left(\mu^2+\left|Du(x)\right|^2\right)^\frac{\gamma-2}{2}D\left(\left|Du(x)\right|^2\right)\otimes Du(x)\right>dx\notag\\
&+\int_{B_R}\left<D_{x\xi}f\left(x, Du(x)\right), \eta^2(x) \left(\mu^2+\left|Du(x)\right|^2\right)^\frac{\gamma}{2} D^2u(x)\right>dx\notag\\
&+2\int_{B_R}\left<D_{\xi\xi}f\left(x, Du(x)\right)D^2u(x), \eta(x)\left(\mu^2+\left|Du(x)\right|^2\right)^\frac{\gamma}{2}D\eta(x)\otimes Du(x)\right>dx\notag\\
&+\frac{\gamma}{2}\int_{B_R}\left<D_{\xi\xi}f\left(x, Du(x)\right)D^2u(x), \eta^2(x)\left(\mu^2+\left|Du(x)\right|^2\right)^\frac{\gamma-2}{2}D\left(\left|Du(x)\right|^2\right)\otimes Du(x)\right>dx\notag\\
&+\int_{B_R}\left<D_{\xi\xi}f\left(x, Du(x)\right)D^2u(x),\eta^2(x)\left(\mu^2+\left|Du(x)\right|^2\right)^\frac{\gamma}{2}D^2u(x)\right>dx\notag\\
&=I+II+III+IV+V+I_0.
\end{align}

The integral $V$ is non-negative by the assumption $f(x, \xi)=k(x, |\xi|)$. Actually, it suffices to calculate 

\begin{equation*}
D_{{\xi^\alpha_i}{\xi^\beta_j}}f(x, \xi)=D_{tt}k(x, \left|\xi\right|)\frac{\xi^\alpha_i\xi^\beta_j}{|\xi|^2}+D_tk(x, |\xi|)\left(\frac{\delta^{\alpha\beta}\delta_{ij}}{\left|\xi\right|}-\frac{\xi^\alpha_i\xi^\beta_j}{\left|\xi\right|^3}\right)
\end{equation*}

and use the definition of the scalar product to deduce that

\begin{equation}\label{V}
V\ge\frac{\gamma}{2}\int_{B_{R}}\left(\mu^2+\left|Du(x)\right|^2\right)^\frac{p-2+\gamma}{2}D\left(\left|Du(x)\right|^2\right)dx \ge 0.
\end{equation}

So, from \eqref{SecondVariation3}, we get

\begin{equation}\label{generalestimate}
I_0\le I_0+V\le\left|I\right|+\left|II\right|+\left|III\right|+\left|IV\right|.
\end{equation}

In the following, we will often use the trivial inequality

\begin{equation}\label{Dualpha}
\left|\xi\right|\le\left(\mu^2+\left|\xi\right|^2\right)^\frac{1}{2}, \qquad \forall \xi\in\R^{N\times n}.
\end{equation}

By the left inequality in the hypothesis \eqref{SecondDerivativeGrowth}, we get 

\begin{align}\label{I_0}
\left|I_0\right|&\ge c\int_{B_R}\eta^2(x)\left(\mu^2+\left|Du(x)\right|^2\right)^\frac{p-2}{2}\left|D^2u(x)\right|^2\left(\mu^2+\left|Du(x)\right|^2\right)^\frac{\gamma}{2}dx\notag\\
&=c\int_{B_R}\eta^2(x)\left|D^2u(x)\right|^2\left(\mu^2+\left|Du(x)\right|^2\right)^\frac{p+\gamma-2}{2}dx.
\end{align}

To estimate the term $I$, we use \eqref{f4} and \eqref{Dualpha}, thus getting

\begin{align}\label{I}
\left|I\right|&\le 2\int_{B_R}\eta(x)\left|D\eta(x)\right|g(x)\left(\mu^2+\left|Du(x)\right|^2\right)^\frac{p-1}{2}\left|Du(x)\right|\left(\mu^2+\left|Du(x)\right|^2\right)^\frac{\gamma}{2} dx\notag\\
&\le 2\int_{B_R}\eta(x)\left|D\eta(x)\right|g(x)\left(\mu^2+\left|Du(x)\right|^2\right)^\frac{p+\gamma}{2}dx.
\end{align}

By Young's Inequality, we have

\begin{align}\label{IY}
&|I|\le2\int_{B_R}\eta(x)\left|D\eta(x)\right|g(x)\left(\mu^2+\left|Du(x)\right|^2\right)^\frac{p+\gamma}{2}\notag\\&\le c\int_{B_R}\eta^2(x)g^2(x)\left(\mu^2+\left|Du(x)\right|^2\right)^\frac{p+\gamma}{2}dx\notag\\
&+c\int_{B_R}\left|D\eta(x)\right|^2\left(\mu^2+\left|Du(x)\right|^2\right)^\frac{p+\gamma}{2}dx.
\end{align}

We use again \eqref{f4} and \eqref{Dualpha} to estimate the term $\left|II\right|$ as follows

\begin{equation}\label{II'}
\left|II\right|\le \gamma\int_{B_R}\eta^2(x)g(x)\left(\mu^2+\left|Du(x)\right|^2\right)^\frac{p-1+\gamma}{2}\left|D^2u(x)\right|dx.
\end{equation}

Writing $\frac{p-1}{2}=\frac{p-2}{4}+\frac{p}{4}$, and using Young's Inequality with exponents $\left(2, 2\right)$, we get

\begin{align}\label{II}
\left|II\right|&\le\varepsilon\int_{B_R}\eta^2(x)\left|D^2u(x)\right|^2\left(\mu^2+\left|Du(x)\right|^2\right)^\frac{p+\gamma-2}{2}dx\notag\\
&+c(\varepsilon)\cdot\gamma^2\int_{B_R}\eta^2(x)g^2(x)\left(\mu^2+\left|Du(x)\right|^2\right)^\frac{p+\gamma}{2}dx.
\end{align}

In order to estimate $\left|III\right|$, we use \eqref{f4} and Young's Inequality as before:

\begin{align}\label{III}
\left|III\right|&\le\int_{B_R}\eta^2(x)g(x)\left|D^2u(x)\right|\left(\mu^2+\left|Du(x)\right|^2\right)^\frac{p+\gamma-1}{2}dx\notag\\
&\le\varepsilon\int_{B_R}\eta^2(x)\left|D^2u(x)\right|^2\left(\mu^2+\left|Du(x)\right|^2\right)^\frac{p+\gamma-2}{2}\notag\\
&+c(\varepsilon)\int_{B_R}\eta^2(x)g^2(x)\left(\mu^2+\left|Du(x)\right|^2\right)^\frac{p+\gamma}{2}.
\end{align}

We can estimate $IV$ using \eqref{SecondDerivativeGrowth} and \eqref{Dualpha} thus getting

\begin{equation}\label{IV'}
\left|IV\right|\le 2\int_{B_R}\eta(x)\left|D\eta(x)\right|\left|D^2u(x)\right|\left(\mu^2+\left|Du(x)\right|^2\right)^\frac{p+\gamma-1}{2}dx.
\end{equation}

Since $\frac{\gamma+1}{2}=\frac{\gamma}{4}+\frac{\gamma+2}{4}$, using Young's Inequality, we have

\begin{align}\label{IV}
\left|IV\right|&\le\varepsilon\int_{B_R}\eta^2(x)\left|D^2u(x)\right|^2\left(\mu^2+\left|Du(x)\right|^2\right)^\frac{p+\gamma-2}{2}dx\notag\\
&+c(\varepsilon)\int_{B_R}\left|D\eta(x)\right|^2\left(\mu^2+\left|Du(x)\right|^2\right)^\frac{p+\gamma}{2}dx.
\end{align}

Now, inserting \eqref{I_0}, \eqref{IY}, \eqref{II}, \eqref{III} and \eqref{IV} in \eqref{generalestimate}, and choosing $\varepsilon$ such that we can reabsorb the first terms on the right-hand sides of \eqref{II} and \eqref{III}, we get

\begin{align}\label{-}
&\int_{B_R}\eta^2(x)\left(\mu^2+\left|Du(x)\right|^2\right)^\frac{p+\gamma-2}{2}\left|D^2u(x)\right|^2dx\notag\\
&\le c(1+\gamma^2)\int_{B_R}\eta^2(x)g^2(x)\left(\mu^2+\left|Du(x)\right|^2\right)^\frac{p+\gamma}{2}dx\notag\\
&+c\int_{B_R}\left|D\eta(x)\right|^2\left(\mu^2+\left|Du(x)\right|^2\right)^\frac{p+\gamma}{2}dx.
\end{align} 

We want to control the first integral in the right-hand side of \eqref{-} with some terms like the others of the same inequality.\\

Recalling the definition of the auxiliary function $H$, in  \eqref{Hdef}, \eqref{-} becomes

\begin{align}\label{2.17bis}
&\int_{B_R}\eta^2(x)H^{p+\gamma-2}\left(Du(x)\right)\left|D^2u(x)\right|^2dx\notag\\
&\le c(1+\gamma^2)\int_{B_R}\eta^2(x)g^2(x)H^{p+\gamma}\left(Du(x)\right)dx\notag\\
&+c\int_{B_R}\left|D\eta(x)\right|^2H^{p+\gamma}\left(Du(x)\right)dx.
\end{align} 

Now, we observe that

\begin{equation}\label{**'}
H^{p+\gamma-4}\left(Du\right)\cdot\left|D\left(\left|Du\right|^2\right)\right|^2=4H^{p+\gamma-4}\left(Du\right)\left|Du\right|^2\left|D^2u\right|^2\le 4 H^{p+\gamma-2}\left|D^2u\right|^2,
\end{equation}

where we also used \eqref{Dualpha}. So, using \eqref{**'} in the left-hand side of \eqref{2.17bis}, we get

\begin{align}\label{***''}
&\int_{B_R}\eta^2(x)H^{p+\gamma-4}\left(Du(x)\right)\left|D\left(\left|Du(x)\right|^2\right)\right|^2dx\notag\\
&\le c(1+\gamma^2)\int_{B_R}\eta^2(x)g^2(x)H^{p+\gamma}\left(Du(x)\right)dx\notag\\
&+c\int_{B_R}\left|D\eta(x)\right|^2H^{p+\gamma}\left(Du(x)\right)dx.
\end{align}

One can easily check that, for any $\alpha\in\R$, 

\begin{equation}\label{staralpha}
D\left(H^\alpha\left(Du\right)\right)=\frac{\alpha}{2}\cdot H^{\alpha-2}\left(Du\right)\cdot D\left(\left|Du\right|^2\right),
\end{equation}

So, using \eqref{staralpha} with $\alpha=\frac{p+\gamma}{2}$, we have

\begin{align}\label{starp}
&\left|H^{p+\gamma-4}\left(Du\right)\cdot\left|D\left(\left|Du\right|^2\right)\right|^2\right|=\left|H^{\frac{p+\gamma}{2}-2}\left(Du\right)\cdot\left|D\left(\left|Du\right|^2\right)\right|\right|^2\notag\\
&=\left|\frac{4}{p+\gamma}\cdot D\left(H^{\frac{p+\gamma}{2}}\left(Du\right)\right)\right|^2.
\end{align}

Combining \eqref{starp} with \eqref{***''}, we get

\begin{align}\label{***''+starp}
&\frac{4}{p+\gamma}\int_{B_{R}}\eta^2(x)\left|\cdot D\left(H^{\frac{p+\gamma}{2}}\left(Du(x)\right)\right)\right|^2dx\notag\\&\le\int_{B_R}\eta^2(x)H^{p+\gamma-4}\left(Du(x)\right)\left|D\left(\left|Du(x)\right|^2\right)\right|^2dx\notag\\
&\le c(1+\gamma^2)\int_{B_R}\eta^2(x)g^2(x)H^{p+\gamma}\left(Du(x)\right)dx\notag\\
&+c\int_{B_R}\left|D\eta(x)\right|^2H^{p+\gamma}\left(Du(x)\right)dx.
\end{align}

Before going further, since we want to apply Moser's iteration technique, starting from $\gamma=0$, let's observe that, if $u\in W^{2, 2}_{\mathrm{loc}}(\Omega)\cap W^{1, \infty}_{\mathrm{loc}}(\Omega)$, then $H^{\frac{p+\gamma}{2}}\left(Du\right)\in W^{1,2}_{\mathrm{loc}}(\Omega)$.

Now we define the function

\begin{equation}\label{G-definition}
G=\eta\cdot H^{\frac{p+\gamma}{2}}\left(Du\right),
\end{equation}

so that, for $\gamma=0$, $G\in W^{1,2}_{0}(B_R)$, and denoting $2^*=\frac{2n}{n-2}$, by Sobolev's Inequality we have

\begin{equation}\label{Sobolev-G}
\left(\int_{B_R}\left|G(x)\right|^{2^*}dx\right)^\frac{2}{2^*}\le c\int_{B_{R}}\left|DG(x)\right|^2dx,
\end{equation}

so

\begin{align}\label{****'}
&\left(\int_{B_{R}}\eta^{2^*}(x)\cdot H^{\frac{2^*}{2}\cdot(p+\gamma)}\left(Du(x)\right)dx\right)^{\frac{2}{2^*}}\notag\\
&\le c\left(\int_{B_{R}}\Bigg|\eta(x)\cdot \left|D\left(H^{\frac{p+\gamma}{2}}\left(Du(x)\right)\right)\right|+\left|D\eta(x)\right|\cdot H^{\frac{p+\gamma}{2}}\left(Du(x)\right)\Bigg|^2dx\right)\notag\\
&\le c\int_{B_{R}}\eta^2(x)\cdot \left|D\left(H^{\frac{p+\gamma}{2}}\left(Du(x)\right)\right)\right|^2dx+c\int_{B_{R}}\left|D\eta(x)\right|^2\cdot H^{p+\gamma}\left(Du(x)\right)dx.
\end{align}

Joining \eqref{****'} with \eqref{***''+starp}, we get

\begin{align}\label{canc}
&\left(\int_{B_{R}}\eta^{2^*}(x)\cdot H^{\frac{2^*}{2}\cdot(p+\gamma)}\left(Du(x)\right)dx\right)^{\frac{2}{2^*}}\notag\\
&\le c\left(\frac{p+\gamma}{4}\right)^2\cdot\int_{B_R}\eta^2(x)\cdot H^{p+\gamma-4}\left(Du(x)\right)\left|D\left(\left|Du(x)\right|^2\right)\right|^2dx\notag\\
&+c\int_{B_{R}}\left|D\eta(x)\right|^2\cdot H^{p+\gamma}\left(Du(x)\right)dx\notag\\
&\le c\left(\frac{p+\gamma}{4}\right)^2\cdot\left[(1+\gamma^2)\int_{B_R}\eta^2(x)g^2(x)H^{p+\gamma}\left(Du(x)\right)dx\right.\notag\\
&\left.\qquad+\int_{B_R}\left|D\eta(x)\right|^2H^{p+\gamma}\left(Du(x)\right)dx\right]\notag\\
&+c\int_{B_{R}}\left|D\eta(x)\right|^2\cdot H^{p+\gamma}\left(Du(x)\right)dx\notag\\
&=c\left(\frac{p+\gamma}{4}\right)^2\cdot(1+\gamma^2)\cdot\int_{B_R}\eta^2(x)g^2(x)H^{p+\gamma}\left(Du(x)\right)dx\notag\\
&+c\left[1+\left(\frac{p+\gamma}{4}\right)^2\right]\cdot\int_{B_{R}}\left|D\eta(x)\right|^2\cdot H^{p+\gamma}\left(Du(x)\right)dx\notag\\
&\le c\left(p+\gamma\right)^4\cdot\int_{B_R}\eta^2(x)g^2(x)H^{p+\gamma}\left(Du(x)\right)dx\notag\\
&+c\left[1+\left(p+\gamma\right)^2\right]\cdot\int_{B_{R}}\left|D\eta(x)\right|^2\cdot H^{p+\gamma}\left(Du(x)\right)dx.
\end{align}

Now, recalling that $g\in L^q_{\mathrm{loc}}(\Omega)$, with $q>n>2$, we can use H\"{o}lder's Inequality with exponents $\left(\frac{q}{2}, \frac{q}{q-2}\right)$, and we infer

\begin{align}\label{g-Holder}
&\int_{B_R}\eta^2(x)g^2(x)H^{p+\gamma}\left(Du(x)\right)dx\notag\\
&\le\left(\int_{B_{R}}g^q(x)dx\right)^\frac{2}{q}\cdot\left(\int_{B_{R}}\eta^{\frac{2q}{q-2}}(x)H^{\frac{q(p+\gamma)}{q-2}}\left(Du(x)\right)dx\right)^\frac{q-2}{q}.
\end{align}

Since $q>n$, $1<\frac{q}{q-2}<\frac{n}{n-2}$, and we can apply the Interpolation Inequality to the function

\begin{equation}\label{h-definition}
\eta^2\cdot H^{p+\gamma}\left(Du\right).
\end{equation}

Let $\theta\in(0, 1)$ such that

\begin{equation}
\frac{q-2}{q}=\theta+\frac{(1-\theta)(n-2)}{n}.
\end{equation}

One can easily check that

\begin{equation}
\theta=\frac{q-n}{q},
\end{equation}

and so

\begin{align}
&\left[\int_{B_{R}}\left(\eta^2(x)\cdot H^{p+\gamma}\left(Du(x)\right)\right)^{\frac{q}{q-2}}dx\right]^\frac{q-2}{q}\notag\\&\le c\left(\int_{B_{R}}\eta^2(x)\cdot H^{p+\gamma}\left(Du(x)\right)dx\right)^\theta\cdot\left(\int_{B_{R}}\left(\eta^2(x)\cdot H^{p+\gamma}\left(Du(x)\right)\right)^{\frac{n}{n-2}}(x)dx\right)^\frac{(1-\theta)(n-2)}{n},
\end{align}

that is

\begin{align}\label{Interpolation}
&\left(\int_{B_{R}}\eta^{\frac{2q}{q-2}}(x)H^{\frac{q(p+\gamma)}{q-2}}\left(Du(x)\right)dx\right)^\frac{q-2}{q}\notag\\
&\le c\left(\int_{B_{R}}\eta^2(x)H^{p+\gamma}\left(Du(x)\right)dx\right)^\theta\cdot \left(\int_{B_{R}}\eta^{2^*}(x)H^{\frac{2^*}{2}\cdot(p+\gamma)}\left(Du(x)\right)dx\right)^{\frac{2(1-\theta)}{2^*}}.
\end{align}

Using \eqref{g-Holder}, \eqref{Interpolation}, and Young's Inequality with exponents $\left(\frac{1}{\theta}, \frac{1}{1-\theta}\right)$, for any $\varepsilon>0$, we can estimate the first term in the right-hand side of \eqref{canc} as follows

\begin{align}\label{Interpolation-Young}
&c\left(p+\gamma\right)^4\int_{B_R}\eta^2(x)g^2(x)H^{p+\gamma}\left(Du(x)\right)dx\notag\\
&\le c(\varepsilon)\left[c\left(p+\gamma\right)^4\left(\int_{B_{R}}g^q(x)dx\right)^\frac{2}{q}\right]^{\frac{1}{\theta}}\cdot\left(\int_{B_{R}}\eta^2(x)H^{p+\gamma}\left(Du(x)\right)dx\right)\notag\\
&+\varepsilon\left(\int_{B_{R}}\eta^{2^*}(x)H^{2^*\cdot\frac{p+\gamma}{2}}\left(Du(x)\right)dx\right)^{\frac{2}{2^*}}.
\end{align}

Now, plugging \eqref{Interpolation-Young} into \eqref{canc}, we get

\begin{align}\label{beforereabsorbing}
&\left(\int_{B_{R}}\eta^{2^*}(x)\cdot H^{\frac{2^*}{2}\cdot(p+\gamma)}\left(Du(x)\right)dx\right)^{\frac{2}{2^*}}\notag\\
&\le c\left(p+\gamma\right)^4\int_{B_R}\eta^2(x)g^2(x)H^{p+\gamma}\left(Du(x)\right)dx\notag\\
&+c\left[1+\left(p+\gamma\right)^2\right]\int_{B_{R}}\left|D\eta(x)\right|^2\cdot H^{p+\gamma}\left(Du(x)\right)dx\notag\\
&\le\varepsilon\left(\int_{B_{R}}\eta^{2^*}(x)H^{2^*\cdot\frac{p+\gamma}{2}}\left(Du(x)\right)dx\right)^{\frac{2}{2^*}}\notag\\
&+c(\varepsilon)\left[c\left(p+\gamma\right)^4\left(\int_{B_{R}}g^q(x)dx\right)^\frac{2}{q}\right]^{\frac{1}{\theta}}\cdot\left(\int_{B_{R}}\eta^2(x)H^{p+\gamma}\left(Du(x)\right)dx\right)\notag\\
&+c\left[1+\left(p+\gamma\right)^2\right]\cdot\int_{B_{R}}\left|D\eta(x)\right|^2\cdot H^{p+\gamma}\left(Du(x)\right)dx
\end{align}

and reabsorbing for a sufficiently small value of $\varepsilon$, and recalling the explicit expression of $\theta$, we get

\begin{align}\label{canc2}
&\left(\int_{B_{R}}\eta^{2^*}(x)H^{\frac{2^*}{2}\cdot(p+\gamma)}\left(Du(x)\right)dx\right)^{\frac{2}{2^*}}\notag\\
&\le c\left[\left(p+\gamma\right)^4\cdot\left(\int_{B_{R}}g^q(x)dx\right)^{\frac{2}{q}}\right]^{\frac{q}{q-n}}\cdot\left(\int_{B_{R}}\eta^2(x)H^{p+\gamma}\left(Du(x)\right)dx\right)\notag\\
&+c\left[1+\left(p+\gamma\right)^2\right]\cdot\left(\int_{B_{R}}\left|D\eta(x)\right|^2H^{p+\gamma}\left(Du(x)\right)dx\right).
\end{align}

For $\gamma=0$, \eqref{canc2} gives

\begin{align}\label{canc3}
&\left(\int_{B_{R}}\eta^{2^*}(x)H^{\frac{2^*}{2}\cdot p}\left(Du(x)\right)dx\right)^{\frac{2}{2^*}}\notag\\
&\le cp^{\frac{4q}{q-n}}\cdot\left(\int_{B_{R}}g^q(x)dx\right)^{\frac{2}{q-n}}\cdot\left(\int_{B_{R}}\eta^2(x)H^{p}\left(Du(x)\right)dx\right)\notag\\
&+c\left(1+p^2\right)\cdot\left(\int_{B_{R}}\left|D\eta(x)\right|^2H^{p}\left(Du(x)\right)dx\right).
\end{align}

Since, by the absolute continuity of the integral, there is $\bar{R}>0$ such that, if $R<\bar{R}$, then

\begin{equation}\label{abs}
\left(\int_{B_{R}}g^q(x)\right)^{\frac{2}{q}}<1,
\end{equation}
	
recalling the properties of $\eta$, we can write

\begin{equation}\label{Iteration0}
\left(\int_{B_{\rho}}H^{p\cdot\frac{2^*}{2}}\left(Du(x)\right)dx\right)^{\frac{2}{2^*}}\le \frac{c\cdot p^{\frac{4q}{q-n}}}{(R-\rho)^2}\int_{B_{R}}H^p\left(Du(x)\right)dx,
\end{equation}

where $c=c(n, N, p, q, L_1, L_2, \Arrowvert g\Arrowvert_{L^q(B_R)})$.

Now we choose $\rho=\frac{R}{2}$ and set

\begin{equation}\label{settingiteration1}
R_0=R,\qquad R_i=\rho+\frac{R-\rho}{2^i}=\frac{R}{2}\left(1+\frac{1}{2^i}\right), \qquad \forall i\in\N
\end{equation}

and

\begin{equation}\label{settiniteration2}
p_0=p,\qquad p_i=\frac{2^*}{2}\cdot p_{i-1}=\left(\frac{2^*}{2}\right)^i\cdot p_0, \qquad \forall i\in\N.
\end{equation}

Observe that the sequence $R_i$ is strictly decreasing, and $p_i$ is strictly increasing. Moreover, as $i\to\infty$, $R_i\to\frac{R}{2}$ and $p_i\to\infty$.

Starting from \eqref{canc3}, we can iterate \eqref{canc2}, thus getting, for every $i\in\N$, since $\frac{2}{2^*}=\frac{p_i}{p_{i+1}}$, and $R_i-R_{i+1}=\frac{R}{2^{i+2}}$,

\begin{align}\label{Iteration1}
&\left(\int_{B_{R_{i+1}}}H^{p_{i+1}}\left(Du(x)\right)dx\right)^{\frac{1}{p_{i+1}}}\le \left[\frac{c\cdot p_i^{\frac{4q}{q-n}}}{\left(R_i-R_{i+1}\right)^2}\right]^\frac{1}{p_i}\left(\int_{B_{R_{i}}}H^{p_{i}}\left(Du(x)\right)dx\right)^{\frac{1}{p_i}}\notag\\
&\le\prod_{k=0}^{i}\left(\left[\frac{cp_k^{\frac{4q}{q-n}}}{\left(R_k-R_{k+1}\right)^2}\right]^\frac{1}{p_k}\right)\cdot\left(\int_{B_{R}}H^{p}\left(Du(x)\right)dx\right)^{\frac{1}{p}}\notag\\
&=\exp\left\lbrace \sum_{k=0}^{i}\left[\frac{1}{p_k}\cdot\log\left(\frac{c\cdot 2^{k+2}p_k^{\frac{4q}{q-n}}}{R^2}\right)\right]\right\rbrace \cdot\left(\int_{B_{R}}H^{p}\left(Du(x)\right)dx\right)^{\frac{1}{p}}.
\end{align}
 
Since the series

\begin{equation}\label{series}
\sum_{k=0}^{\infty}\left\lbrace\frac{1}{p_k}\cdot\log\left(\frac{c\cdot 2^{k+2}p_k^{\frac{4q}{q-n}}}{R^2}\right)\right\rbrace
\end{equation}

converges, we can pass to the limit as $i\to\infty$ in \eqref{Iteration1}, thus obtaining the following estimate

\begin{equation}\label{aprioriLinfty}
\Arrowvert H\left(Du\right)\Arrowvert_{L^\infty\left(B_{\frac{R}{2}}\right)}\le c \Arrowvert H\left(Du\right)\Arrowvert_{L^{p}\left(B_{R}\right)}
\end{equation}

where $c=c\left(n, N, p, q, L_1, L_2, \Arrowvert g\Arrowvert_{L^q(B_R)}\right)$, i.e. \eqref{aprioriLinftyclaim}.\\

Moreover, by \eqref{starp} and \eqref{***''} for $\gamma=0$, we get

\begin{align}\label{L2estimate1}
&\int_{B_{R}}\eta^2(x)\left|D\left(H^{\frac{p}{2}}\left(Du(x)\right)\right)\right|^2dx\le\notag\\
&\le c\cdot p^2\left[\int_{B_R}\eta^2(x)g^2(x)H^{p}\left(Du(x)\right)dx+c\int_{B_R}\left|D\eta(x)\right|^2H^{p}\left(Du(x)\right)dx\right].
\end{align}

Using \eqref{g-Holder} and \eqref{Interpolation} again, with the same value of $\theta$, for $\gamma=0$, \eqref{L2estimate1} becomes

\begin{align}\label{L2estimate2}
&\int_{B_{R}}\eta^2(x)\left|D\left(H^{\frac{p}{2}}\left(Du(x)\right)\right)\right|^2dx\le c\cdot p^2\left(\int_{B_{R}}\eta^2(x)H^{p}\left(Du(x)\right)dx\right)^\theta\notag\\&\qquad\cdot \left(\int_{B_{R}}\eta^{2^*}(x)H^{\frac{2^*}{2}\cdot p}\left(Du(x)\right)dx\right)^{\frac{2(1-\theta)}{2^*}}\left(\int_{B_{R}}g^q(x)dx\right)^\frac{2}{q}\notag\\
&+c\cdot p^2\int_{B_R}\left|D\eta(x)\right|^2H^{p}\left(Du(x)\right)dx,
\end{align}

and now we use Young's Inequality with exponents $\left(\frac{1}{\theta}, \frac{1}{1-\theta}\right)$, thus obtaining

\begin{align}\label{L2estimate3}
&\int_{B_{R}}\eta^2(x)\left|D\left(H^{\frac{p}{2}}\left(Du(x)\right)\right)\right|^2dx\le cp^2\left[\left(\int_{B_{R}}\eta^2(x)H^{p}\left(Du(x)\right)dx\right)\right.\notag\\&\left.\cdot\left(\int_{B_{R}}g^q(x)dx\right)^\frac{2}{q\theta}+ \left(\int_{B_{R}}\eta^{2^*}(x)H^{p\cdot\frac{2^*}{2}}\left(Du(x)\right)dx\right)^{\frac{2}{2^*}}\right.\notag\\
&\left.+\int_{B_R}\left|D\eta(x)\right|^2H^{p}\left(Du(x)\right)dx\right],
\end{align}

and, by \eqref{canc2} with $\gamma=0$, 

\begin{align}\label{L2estimate4}
&\int_{B_{R}}\eta^2(x)\left|D\left(H^{\frac{p}{2}}\left(Du(x)\right)\right)\right|^2dx\le cp^{\frac{4q}{q-n}}\cdot\left(\int_{B_{R}}g^q(x)dx\right)^{\frac{2}{q-n}}\notag\\&\cdot\left(\int_{B_{R}}\eta^2(x)H^{p}\left(Du(x)\right)dx\right)+c\left(1+p^2\right)\cdot\left(\int_{B_{R}}\left|D\eta(x)\right|^2H^{p}\left(Du(x)\right)dx\right),
\end{align}

where we used that $\theta=\frac{q-n}{q}.$\\
Recalling the properties of $\eta$, and choosing $R$ such that \eqref{abs} holds true, and $\rho=\frac{R}{2}$, we can estimate the $L^2-$norm of the gradient of $H^\frac{p}{2}\left(Du\right)$ as follows

\begin{align}\label{L2estimate}
&\int_{B_{\frac{R}{2}}}\left|D\left(H^{\frac{p}{2}}\left(Du(x)\right)\right)\right|^2dx\le \frac{c}{R^2}\int_{B_{R}}H^{p}\left(Du(x)\right)dx,
\end{align}

where $c=c\left(n, N, p, q, L_1, L_2, \Arrowvert g\Arrowvert_{L^q(B_R)}\right)$.
Since $p<2$, we also have, by H\"{o}lder's Inequality,

\begin{align}\label{L2D2uestimate1}
&\int_{B_{\frac{R}{2}}}\left|D^2u(x)\right|^pdx=\int_{B_{\frac{R}{2}}}\left|D^2u(x)\right|^pH^{\frac{p\cdot(p-2)}{2}}\left(Du(x)\right)\cdot H^{\frac{p\cdot(2-p)}{2}}\left(Du(x)\right)dx\notag\\
&\le\left(\int_{B_{\frac{R}{2}}}\left|D^2u(x)\right|^2H^{p-2}\left(Du(x)\right)dx\right)^\frac{p}{2}\cdot\left(\int_{B_{R}}H^p\left(Du(x)\right)dx\right)^{\frac{2-p}{2}}.
\end{align}

Now we estimate the first integral in the right-hand side of \eqref{L2D2uestimate1} using \eqref{2.17bis} with $\gamma=0$ and \eqref{Iteration0}, so we get

\begin{equation}\label{L2D2uestimate}
\int_{B_{\frac{R}{2}}}\left|D^2u(x)\right|^pdx\le c\cdot\left(\int_{B_{R}}H^p\left(Du(x)\right)dx\right),
\end{equation}

i.e. \eqref{LpD2uestimateclaimqgrn}.
\end{proof}

\subsection{The case $q=n$: proof of Lemma \ref{aprioriqeqn}.}
\begin{proof}[Proof of Lemma \ref{aprioriqeqn}]
Notice that, in this case, we are weakening the assumption on $g$, since $g\in L^n_{\mathrm{loc}}(\Omega)$. As in the previous section, $u$ is a local minimizer of the functional  \eqref{functional}, and we assume $u\in W^{2,2}_{\mathrm{loc}}(\Omega)\cap W^{1, \infty}_{\mathrm{loc}}(\Omega)$.\\
First of all, we find an estimate for the  $L^r-$norm of $H(Du)$, for any $1<r<\infty$, proving \eqref{Lrestimate-claim}.\\
We can argue exactly as the previous case until the estimate \eqref{canc}. In order to estimate the integral \eqref{criticalintegral}, we use H\"{o}lder's Inequality with exponents $\left(\frac{n}{2}, \frac{n}{n-2}\right)$, as follows
\begin{align}\label{g-Holder-n}
&\int_{B_R}\eta^2(x)g^2(x)H^{p+\gamma}\left(Du(x)\right)dx\notag\\
&\le\left(\int_{B_{R}}g^n(x)dx\right)^\frac{2}{n}\cdot\left(\int_{B_{R}}\eta^{2^*}(x)H^{2^*\cdot\frac{(p+\gamma)}{2}}\left(Du(x)\right)dx\right)^\frac{2}{2^*}.
\end{align}

Plugging \eqref{g-Holder-n} into \eqref{canc}, we have

\begin{align}\label{canc-n}
&\left(\int_{B_{R}}\eta^{2^*}(x)\cdot H^{\frac{2^*}{2}\cdot(p+\gamma)}\left(Du(x)\right)dx\right)^{\frac{2}{2^*}}\notag\\
&\le c\left(p+\gamma\right)^4\left(\int_{B_{R}}g^n(x)dx\right)^\frac{2}{n}\cdot\left(\int_{B_{R}}\eta^{2^*}(x)H^{\frac{2^*}{2}\cdot(p+\gamma)}\left(Du(x)\right)dx\right)^\frac{2}{2^*}\notag\\
&+c\left[1+\left(\frac{p+\gamma}{4}\right)^2\right]\int_{B_{R}}\left|D\eta(x)\right|^2\cdot H^{p+\gamma}\left(Du(x)\right)dx.
\end{align}

In order to reabsorb the first term on the right-hand side of \eqref{canc-n}, we have to use the absolute continuity of the integral and take $R<R_{\gamma}$, with $R_\gamma$ such that

\begin{equation}\label{abscont}
\left(\int_{B_{R_\gamma}}g^n(x)dx\right)^\frac{2}{n}<\frac{1}{c\left(p+\gamma\right)^4}.
\end{equation}

Observe that, if $\gamma\to\infty$, then $R_\gamma\to0$,
and so, even if we can still use Moser's Iterative technique, we cannot pass to the limit.\\
More precisely, if $R<R_\gamma$, plugging \eqref{abscont} into \eqref{canc-n}, we can reabsorb the first term of the right-hand side of \eqref{canc-n} to the left-hand side, thus getting

\begin{align}\label{canc-n2}
&\left(\int_{B_{R}}\eta^{2^*}(x)\cdot H^{\frac{2^*}{2}\cdot(p+\gamma)}\left(Du(x)\right)dx\right)^{\frac{2}{2^*}}\notag\\
&\le c\left[1+\left(\frac{p+\gamma}{4}\right)^2\right]\int_{B_{R}}\left|D\eta(x)\right|^2\cdot H^{p+\gamma}\left(Du(x)\right)dx,
\end{align}

and by the properties of $\eta$, for $\gamma=0$ we get

\begin{equation}\label{canc-n3}
\left(\int_{B_{\rho}}H^{\frac{2^*}{2}\cdot p}\left(Du(x)\right)dx\right)^{\frac{2}{2^*}}\le \frac{cp^2}{(R-\rho)^2}\int_{B_{R}}H^{p}\left(Du(x)\right)dx.
\end{equation}

Choosing $\rho=\frac{R}{2}$, by the same iterative method used in the previous proof, recalling \eqref{settingiteration1} and \eqref{settiniteration2}, we get

\begin{align}\label{Iteration-n}
&\left(\int_{B_{R_{i+1}}}H^{p_{i+1}}\left(Du(x)\right)dx\right)^{\frac{1}{p_{i+1}}}\le \left[\frac{c\cdot p_i^{2}}{\left(R_i-R_{i+1}\right)^2}\right]^\frac{1}{p_i}\left(\int_{B_{R_{i}}}H^{p_{i}}\left(Du(x)\right)dx\right)^{\frac{1}{p_i}}\notag\\
&\le\prod_{k=0}^{i}\left(\left[\frac{cp_k^{2}}{\left(R_k-R_{k+1}\right)^2}\right]^\frac{1}{p_k}\right)\cdot\left(\int_{B_{R}}H^{p}\left(Du(x)\right)dx\right)^{\frac{1}{p}}\notag\\
&=\exp\left\lbrace \sum_{k=0}^{i}\left[\frac{1}{p_k}\cdot\log\left(\frac{c\cdot 2^{k+2}p_k^{2}}{R^2}\right)\right]\right\rbrace \cdot\left(\int_{B_{R}}H^{p}\left(Du(x)\right)dx\right)^{\frac{1}{p}},
\end{align}

and since the estimate \eqref{Iteration-n} holds true for every $i\in\N$, and $p_i\to\infty$ as $i\to\infty$, we can estimate the $L^r$ norm of $H(Du)$ for every $1<r<\infty.$ More precisely, for any finite $r$, there is $i\in\N$ such that $p_i>r$, so we have, for a constant $c_1=c_1(r, p, n)$

\begin{align}\label{Lrestimate}
&\left(\int_{B_{\frac{R}{2}}}H^{r}\left(Du(x)\right)dx\right)^{\frac{1}{r}}\le\left(\int_{B_{R_{i+1}}}H^{r}\left(Du(x)\right)dx\right)^{\frac{1}{r}}\notag\\
&\le c_1\left(\int_{B_{R_{i+1}}}H^{p_{i+1}}\left(Du(x)\right)dx\right)^{\frac{1}{p_{i+1}}}\notag\\
&\le c_1\cdot\exp\left\lbrace \sum_{k=0}^{i}\left[\frac{1}{p_k}\cdot\log\left(\frac{c\cdot 2^{k+2}p_k^{2}}{R^2}\right)\right]\right\rbrace \cdot\left(\int_{B_{R}}H^{p}\left(Du(x)\right)dx\right)^{\frac{1}{p}},
\end{align}
and we get \eqref{Lrestimate-claim}.\\

Let us prove, now, estimate \eqref{LpD2uestimateclaimqeqn}. Recalling \eqref{L2estimate1}, using \eqref{g-Holder-n} with $\gamma=0$, we get

\begin{align}\label{L2estimate1-n}
&\int_{B_{R}}\eta^2(x)\left|D\left(H^{\frac{p}{2}}\left(Du(x)\right)\right)\right|^2dx\le c\cdot\left(\frac{p^2}{4}\right)\left[\left(\int_{B_{R}}g^n(x)dx\right)^\frac{2}{n}\right.\notag\\
&\left.\cdot\left(\int_{B_{R}}\eta^{2^*}(x)H^{2^*\cdot\frac{p}{2}}\left(Du(x)\right)dx\right)^\frac{2}{2^*}+c\int_{B_R}\left|D\eta(x)\right|^2H^{p}\left(Du(x)\right)dx\right],
\end{align}

and recalling the properties of $\eta$, with $\rho=\frac{R}{2}$, by \eqref{canc-n3}, we obtain

\begin{align}\label{L2estimate2-n}
&\int_{B_{\frac{R}{2}}}\left|D\left(H^{\frac{p}{2}}\left(Du(x)\right)\right)\right|^2dx\le \frac{cp^2}{R^2}\cdot\left[p^2\cdot\left(\int_{B_{R}}g^n(x)dx\right)^\frac{2}{n}+1\right]\notag\\&\cdot\int_{B_R}H^{p}\left(Du(x)\right)dx.
\end{align}

therefore, using \eqref{abscont}, with $\gamma=0$, we get 

\begin{equation}\label{L2estimate-n}
\int_{B_{\frac{R}{2}}}\left|D\left(H^{\frac{p}{2}}\left(Du(x)\right)\right)\right|^2dx\le \frac{c}{R^2}\cdot\int_{B_R}H^{p}\left(Du(x)\right)dx,
\end{equation}

that is the same a priori estimate as \eqref{L2estimate} under weaker assumption on the coefficients.\\
In a way very similar to \eqref{L2D2uestimate1}, by \eqref{L2estimate-n} using \eqref{2.17bis} and \eqref{canc-n3}, we get the same estimate for the $L^p-$ norm of the second derivatives of $u$, thus getting \eqref{LpD2uestimateclaimqeqn}.
\end{proof} 

\section{Regularity results by approximation: proofs of Theorem \ref{qgrn} and Theorem \ref{qeqn}.}

The aim of this section is to prove that the a priori estimates proved in the section \ref{A priori section} are preserved in passing to the limit in a sequence of minimizers of a suitable approximating problem, and this allows us to prove Theorem \ref{qgrn} and Theorem \ref{qeqn}.

\begin{proof}[Proof of Theorem \ref{qgrn}]
Let us consider a function $\phi\in C^{\infty}_0(B_1(0))$ such that $0\le\phi\le1$ and $\int_{B_1(0)}\phi(x)dx=1$, and, for all $\varepsilon>0$, a standard family of mollifiers $\set{\phi_\varepsilon}_\varepsilon$  defined as follows
	
	\begin{equation*}
		\phi_\varepsilon(x)=\frac{1}{\varepsilon^n}\phi\left(\frac{x}{\varepsilon}\right),
	\end{equation*}
	
	so that, for all $\varepsilon>0$, $\phi_\varepsilon\in C^{\infty}_0(B_\varepsilon(0))$, $0\le\phi_\varepsilon\le1$, $\int_{B_\varepsilon(0)}\phi_\varepsilon(x)dx=1.$
	
	It is well known that, for any $h\in L^1_{\mathrm{loc}}(\Omega)$,  with $d(\text{supp }(h), \partial\Omega)>\varepsilon$ setting 
	\begin{equation*}
		h_\varepsilon(x)=h\ast\phi_\varepsilon(x)=\int_{B_\varepsilon}\phi_\varepsilon(y)h(x+y)dy=\int_{B_1}\phi(\omega)h(x+\varepsilon\omega)d\omega,
	\end{equation*}
	we have $h_\varepsilon\in C^\infty(\Omega)$.
	
	Fix $x_0\in\Omega$, $0<R<d(x_0, \partial\Omega)$, and denote $B_R(x_0)=B_R$. Let us consider the following functional
	
	\begin{equation}\label{approxfunctional}
		\mathcal{F_\varepsilon}(v, B_R)=\int_{B_R}f_\varepsilon\left(x, Dv(x)\right)dx,
	\end{equation}
	
	that is 
	
	\begin{equation*}
		\mathcal{F_\varepsilon}(v, B_R)=\int_{B_R}\left(\int_{B_1}f(x+\varepsilon\omega, Dv(x+\varepsilon\omega))\cdot\phi(\omega)d\omega\right)dx.
	\end{equation*}
	
	Let $u\in W^{1, p}_{\mathrm{loc}}$ a local minimizer of the functional \eqref{functional}, and, for each admissible $\varepsilon>0$, let $v_\varepsilon\in W^{1, p}_{\mathrm{loc}}(B_R)$ the unique local minimizer of the functional \eqref{approxfunctional} such that $v_\varepsilon-u\in W^{1, p}_0(B_R)$.\\
	It's known that $v_\varepsilon\in W^{2, 2}_{\mathrm{loc}}\left(B_R\right)\cap W^{1,\infty}_{\mathrm{loc}}\left(B_R\right)$. It's easy to check that from \eqref{f1}, \eqref{f4} and \eqref{SecondDerivativeGrowth}, the following properties hold for the funcion $f_\varepsilon$:
	
	\begin{equation}\label{f'1}
		L_1(\mu^2+|\xi|^2)^\frac{p}{2}\le f_\varepsilon(x, \xi)\le L_2(\mu^2+|\xi|^2)^\frac{p}{2},
	\end{equation}
	
	\begin{equation}\label{f'4}
		\left|D_x\left(D_\xi f_\varepsilon(x, \xi)\right)\right|\le g_\varepsilon(x)\left(\mu^2+|\xi|^2\right)^\frac{p-1}{2}
	\end{equation} 
	
	\begin{equation}\label{f'2}
	c_1\left(\mu^2+\left|\xi\right|^2\right)^\frac{p-2}{2}\left|\eta\right|^2\le\left<D_{\xi\xi}f_\varepsilon(x, \xi)\eta, \eta\right>\le c_2\left(\mu^2+\left|\xi\right|^2\right)^\frac{p-2}{2}\left|\eta\right|^2,
	\end{equation}

	for all $\xi, \eta\in\R^{N\times n}$, and for almost every $x\in \Omega_{\varepsilon},$ and where $g_\varepsilon=g\ast\phi_\varepsilon.$\\	
	By the growth condition \eqref{f'1}, and the minimality of $v_\varepsilon$, it follows

		\begin{align}\label{+}
			L_1\int_{B_R}\left(\mu^2+\left|Dv_\varepsilon(x)\right|^2\right)^\frac{p}{2}dx&\le\int_{B_R}f_\varepsilon\left(x, Dv_\varepsilon(x)\right)dx\le\int_{B_R}f_\varepsilon\left(x, Du(x)\right)dx\notag\\&\le L_2\int_{B_R}\left(\mu^2+\left|Du(x)\right|^2\right)^\frac{p}{2}dx.
		\end{align}

	Since $u\in W^{1,p}(B_R)$, the sequence $\{v_\varepsilon\}_\varepsilon$ is bounded in $W^{1,p}(B_R)$ and so there is a function $v\in W^{1,p}(B_R)$ such that $v_\varepsilon\rightharpoonup v$ in $W^{1,p}(B_R)$. But since $v_\varepsilon\in W^{2, 2}(B_R)\cap W^{1, \infty}(B_R)$, we can use the estimates \eqref{L2estimate} and \eqref{L2D2uestimate} , and then \eqref{+}, thus getting
	
	\begin{equation}\label{2.51epsilon}
		\int_{B_{\frac{R}{2}}}\left|D\left(H^{\frac{p}{2}}\left(Dv_\varepsilon(x)\right)\right)\right|^2dx\le \frac{c}{R^2}\int_{B_{R}}H^{p}\left(Du(x)\right)dx
	\end{equation}
	
	and, by Lemma \ref{aprioriqgrn},
		
	\begin{equation}\label{2.52epsilon}
	\int_{B_{\frac{R}{2}}}\left|D^2v_\varepsilon(x)\right|^pdx\le c\cdot\left(\int_{B_{R}}H^p\left(Dv_\varepsilon(x)\right)dx\right),
	\end{equation}
	
	with a constant depending on $\Arrowvert g_\varepsilon\Arrowvert_{L^{q}(B_{R})}$.\\
	Let's notice that the function $g_\varepsilon$ strongly converges to $g$ in $L^q$, and we have 
	
	$$\Arrowvert g_\varepsilon \Arrowvert_{L^{q}(B_{R})}\le M \Arrowvert g\Arrowvert_{L^{q}(B_{R})},$$ 
	
	and so \eqref{2.51epsilon} and \eqref{2.52epsilon} hold true with a constant independent of $\varepsilon$.
	So  $\left\lbrace H^\frac{p}{2}(Dv_\varepsilon)\right\rbrace _\varepsilon$ is bounded in $W^{1, 2}(B_R)$, and $\left\lbrace v_\varepsilon\right\rbrace_\varepsilon$ is bounded in $W^{2, p}(B_R)$. Then there exists a function $w\in W^{1,2}_{\mathrm{loc}}(B_R)$,  such that $H^\frac{p}{2}(Dv_\varepsilon)$ weakly converges to $w$ in $W^{1, 2}_{\mathrm{loc}}(B_R)$ as $\varepsilon$ goes to 0, so that $H^\frac{p}{2}(Dv_\varepsilon)\to w$ in $L^{2}(B_R)$ strongly, and, up to a subsequence, almost everywhere. Since, by \eqref{2.52epsilon}, $\left\lbrace v_\varepsilon\right\rbrace_\varepsilon$, is bounded in $W^{2, p}(B_R)$, then, up to a subsequence, $v_\varepsilon\rightharpoonup v$ in $W^{2, p}(B_R)$, so $v_\varepsilon\to v$ strongly in $W^{1, p}(B_R)$.\\
	Moreover, since the function $H^\frac{p}{2}$ is continuous, we get
	
	\begin{equation}\label{+''}
	w=H^\frac{p}{2}\left(Dv\right)
	\end{equation}
	
	almost everywhere, and by \eqref{2.51epsilon}, we get
	
	\begin{equation}\label{+'''}
	\int_{B_{\frac{R}{2}}}\left|D\left(H^{\frac{p}{2}}\left(Dv(x)\right)\right)\right|^2dx\le \frac{c}{R^2}\int_{B_{R}}H^{p}\left(Du(x)\right)dx
	\end{equation}
	
	and by \eqref{2.52epsilon}
	
	\begin{equation}\label{+''''}
	\int_{B_{\frac{R}{2}}}\left|D^2v(x)\right|^pdx\le c\cdot\left(\int_{B_{R}}H^p\left(Dv(x)\right)dx\right).
	\end{equation}
	
	Now we want to prove that $u=v$ almost everywhere. Using the minimizing property of $u$ for $\mathcal{F}$, Fatou's Lemma, the lower semi-continuity of $\mathcal{F_\varepsilon}$ (due to the convexity of $f_\varepsilon$), and the fact that $v_\varepsilon$ is the minimizer of $\mathcal{F_\varepsilon}$ with boundary value $u$ on $B_R$, we have

		\begin{align}\label{++}
			&\int_{B_R}f\left(x, Du(x)\right)dx\le\int_{B_R}f\left(x, Dv(x)\right)dx\le\liminf_\varepsilon\int_{B_R}f_\varepsilon\left(x, Dv(x)\right)dx\notag\\&\le\liminf_\varepsilon\int_{B_R}f_\varepsilon\left(x, Dv_\varepsilon(x)\right)dx\le\liminf_\varepsilon\int_{B_R}f_\varepsilon\left(x, Du(x)\right)dx=\int_{B_R}f\left(x, Du(x)\right)dx.
		\end{align}

	So all the terms of \eqref{++} are equal, and in particular
	
	\begin{equation*}
		\int_{B_R}f\left(x, Du(x)\right)dx=\int_{B_R}f\left(x, Dv(x)\right)dx.
	\end{equation*}
	
	By virtue of the strict convexity of the functional \eqref{functional}, the local minimizer with boundary value $u$, is unique, so $u=v$ almost everywhere and $u\in W^{2,p}(B_R)$.\\
	By \eqref{+''''}, we also obtain the following estimate
	
	\begin{equation}\label{D^2uestimate}
	\int_{B_{\frac{R}{2}}}\left|D^2u(x)\right|^2dx\le c\cdot\left(\int_{B_{R}}H^p\left(Du(x)\right)dx\right),
	\end{equation}

that is \eqref{LpD2uestimateclaimqgrnthm}.\\
Now, applying \eqref{aprioriLinftyclaim} to $v_\varepsilon$ and recalling \eqref{+}, we get
	
	\begin{equation}
	\Arrowvert H\left(Dv_\varepsilon\right)\Arrowvert_{L^\infty\left(B_{\frac{R}{2}}\right)}\le c \Arrowvert H\left(Dv_\varepsilon\right)\Arrowvert_{L^{p}\left(B_{R}\right)}\le c \Arrowvert H\left(Du\right)\Arrowvert_{L^{p}\left(B_{R}\right)},
	\end{equation}
	
	and so there is a function $\bar{w}\in W^{1, \infty}(B_R)$ such that $H\left(Dv_\varepsilon\right)\rightharpoonup\bar{w}$ in $W^{1, \infty}(B_R)$, so $H\left(Dv_\varepsilon\right)\to\bar{w}$ in $L^{\infty}(B_R)$, and, as before, by the continuity of $H$, we get $\bar{w}=H(Dv)=H(Du)$.
	By the lower semicontinuity of the $W^{1,\infty}-$norm, we get
	
	\begin{align}\label{Linftyestimate}
	&\Arrowvert  H\left(Du\right)\Arrowvert_{L^\infty\left(B_{\frac{R}{2}}\right)}\le\liminf_\varepsilon\Arrowvert  H\left(Dv_\varepsilon\right)\Arrowvert_{L^\infty\left(B_{\frac{R}{2}}\right)}\notag\\
	&\le c\cdot\liminf_\varepsilon\Arrowvert H\left(Dv_\varepsilon\Arrowvert_{L^{p}\left(B_{R}\right)}\right)\le c \Arrowvert H\left(Du\right)\Arrowvert_{L^{p}\left(B_{R}\right)},
	\end{align}

so we have $H(Du)\in L^\infty_{\mathrm{loc}}(\Omega)$, with the estimate \eqref{Linftyclaim}.
\end{proof}

\begin{proof}[Proof of Theorem \ref{qeqn}]                
In order to prove Theorem \ref{qeqn}, let us observe that, by the same arguments given above, we immediately obtain that $u\in W^{2, p}_{\mathrm{loc}}(\Omega)$, with the estimate \eqref{LpD2uestimateclaimqeqnthm}.\\
To prove the remaining part of the theorem, for $1<r<\infty$, using \eqref{Lrestimate} and \eqref{+}, we have 

\begin{align}\label{Lrestimateregular}
\Arrowvert H\left(Dv_\varepsilon\right)\Arrowvert_{L^r\left(B_{\frac{R}{2}}\right)}\le c \Arrowvert H\left(Dv_\varepsilon\right)\Arrowvert_{L^{p}\left(B_{R}\right)}\le c \Arrowvert H\left(Du\right)\Arrowvert_{L^{p}\left(B_{R}\right)}.
\end{align}

Arguing similarly to how we did for \eqref{Linftyestimate}, we get

\begin{align}\label{LRestimate}
&\Arrowvert  H\left(Du\right)\Arrowvert_{L^r\left(B_{\frac{R}{2}}\right)}\le\liminf_\varepsilon\Arrowvert  H\left(Dv_\varepsilon\right)\Arrowvert_{L^r\left(B_{\frac{R}{2}}\right)}\notag\\
&\le c\cdot\liminf_\varepsilon\Arrowvert H\left(Dv_\varepsilon\Arrowvert_{L^{p}\left(B_{R}\right)}\right)\le c \Arrowvert H\left(Du\right)\Arrowvert_{L^{p}\left(B_{R}\right)}.
\end{align}

So $H(Du)\in L^r_{\mathrm{loc}}(\Omega),$ and estimate \eqref{Lrestimate-claimthm} holds, for every $r\in(1, \infty).$
\end{proof}

	{Andrea Gentile
	\\
	Dipartimento di Matematica e Applicazioni ``R. Caccioppoli''
	\\
	Universit\`a degli Studi di Napoli ``Federico II''\\
	Via Cintia, 80126, Napoli (Italy)
	\\andrea.gentile@unina.it}
\end{document}